\documentclass[hidelinks, 11pt, reqno]{amsart}

\usepackage{amsmath, amssymb, amsthm, xcolor, xspace, hyperref, enumerate}

\newtheorem{thm}{Theorem}
\newtheorem{lem}[thm]{Lemma}

\title{On Viazovska's modular form inequalities}

\author{Dan Romik}
\address{Department of Mathematics \\ University of California, Davis \\ One Shields Ave \\ Davis CA 95616}
\email{romik@math.ucdavis.edu}
\subjclass[2000]{11F11, 26D07}
\keywords{Jacobi thetanull function, Eisenstein series, modular form, inequality, sphere packing}

\begin{document}

\maketitle

\vspace{-10pt}
\begin{abstract}
Viazovska proved that the $E_8$ lattice sphere packing is the densest sphere packing in $8$ dimensions. Her proof relies on two inequalities between functions defined in terms of modular and quasimodular forms. We give a direct proof of these inequalities that does not rely on computer calculations.
\end{abstract}

\section{Introduction}

Viazovska \cite{viazovska} proved that the sphere packing associated with the $E_8$ lattice, which has a packing density of $\frac{\pi^4}{384}$, is the densest sphere packing in $8$ dimensions. Her proof relied on properties of certain functions, denoted $\phi_0(z)$ and $\psi_I(z)$, which were defined in terms of classical modular and quasimodular forms: the Eisenstein series $E_2$, $E_4$ and $E_6$, and the Jacobi thetanull functions $\theta_2$, $\theta_3$ and $\theta_4$. A key step in the proof consisted of showing that these functions satisfied a certain pair of inequalities; this was essential to verifying that a radial function defined by taking an integral transform of $\phi_0(z)$ and $\psi_I(z)$ (combined in a particular way) was the so-called \emph{magic function} that had been conjectured to exist by Cohn and Elkies \cite{cohn-elkies}
and certifies the correct sphere packing bound.

The goal of this paper is to give a new and direct proof of Viazovska's inequalities.
To recall the result, let $z$ denote a complex variable taking values in the upper half plane, and denote $q=e^{\pi i z}$. Let $\sigma_\alpha(n) = \sum_{d\,|\,n} d^\alpha$ denote the divisor function. Recall the definitions of the functions $E_2$, $E_4$, $E_6$, $\theta_2$, $\theta_3$ and $\theta_4$:
\begin{equation*}
\begin{array}{rclcrcl}
\displaystyle E_2(z)  \hspace{-6pt}& = &\hspace{-6pt} \displaystyle 1 - 24 \sum_{n=1}^\infty \sigma_1(n) q^{2n},
& & \theta_2(z) \hspace{-6pt}& = &\hspace{-6pt} \displaystyle \sum_{n=-\infty}^\infty q^{(n+1/2)^2},
\\[12pt]
\displaystyle E_4(z) \hspace{-6pt}& = &\hspace{-6pt} \displaystyle 1 + 240 \sum_{n=1}^\infty \sigma_3(n) q^{2n},
& & \theta_3(z) \hspace{-6pt}& = &\hspace{-6pt} \displaystyle  \sum_{n=-\infty}^\infty q^{n^2},
\\[12pt]
\displaystyle E_6(z) \hspace{-6pt}& = &\hspace{-6pt} \displaystyle 1 - 504 \sum_{n=1}^\infty \sigma_5(n) q^{2n},
& & \theta_4(z) \hspace{-6pt}& = &\hspace{-6pt} \displaystyle \sum_{n=-\infty}^\infty (-1)^n q^{n^2}.
\end{array}
\end{equation*}
Next, set
\begin{align}
\phi_0(z) & = 1728 \frac{(E_2(z) E_4(z) - E_6(z))^2}{E_4(z)^3 - E_6(z)^2},
\label{eq:def-phi0}
\\
\psi_I(z) & = 128 \left( \frac{\theta_3(z)^4 + \theta_4(z)^4}{\theta_2(z)^8} + 
\frac{\theta_4(z)^4 - \theta_2(z)^4}{\theta_3(z)^8} \right),
\label{eq:def-psiI}
\end{align}
and define functions $A(t)$, $B(t)$ of a real variable $t>0$ by
\begin{align*}
A(t) & = -t^2 \phi_0(i/t) - \frac{36}{\pi^2} \psi_I(it), 
\\
B(t) & = -t^2 \phi_0(i/t) + \frac{36}{\pi^2} \psi_I(it).
\end{align*}

\begin{thm}[Viazovska's modular form inequalities]
\label{thm:viaz-ineqs}
The functions $A(t)$, $B(t)$ satisfy
\begin{align}
\label{viaz-ineq1} \tag{V1}
A(t) < 0 \qquad (t>0),
\\
\label{viaz-ineq2} \tag{V2}
B(t) > 0  \qquad (t>0).
\end{align}
\end{thm}

Viazovska's original proof of Theorem~\ref{thm:viaz-ineqs} relied heavily on computer calculations. The proof consisted of two main steps: first, analogues of the inequalities \eqref{viaz-ineq1}--\eqref{viaz-ineq2} were verified numerically for approximating functions $A_0^{(6)}(t)$, $A_\infty^{(6)}(t)$, $B_0^{(6)}(t)$, $B_\infty^{(6)}(t)$ of $A(t)$ and $B(t)$, which were formed by truncating the asymptotic expansions of $A(t)$ and $B(t)$ near $t=0$ and $t=\infty$; this could be done in a finite calculation. Second, rigorous bounds were derived that made it possible to deduce the inequalities \eqref{viaz-ineq1}--\eqref{viaz-ineq2} from the corresponding inequalities for the approximating functions.

Another pair of inequalities of similar flavor to \eqref{viaz-ineq1}--\eqref{viaz-ineq2} was proved by Cohn et al \cite{cohnetal24} in their subsequent proof of optimality of the Leech lattice packing in 24 dimensions. Their proof used different techniques, but that proof as well remained dependent on extensive computer calculations.

Below, we give a new proof of Theorem~\ref{thm:viaz-ineqs} that is fully human-verifiable and requires no numerical calculations beyond the elementary manipulation of a few standard mathematical constants. This helps to simplify and demystify a critical step in Viazovska's celebrated sphere packing proof.

\section{Proof of \eqref{viaz-ineq1}}

It is sufficient to prove that $\phi_0(it) > 0$ and $\psi_I(it) > 0$ for all $t>0$. 
The first of these claims follows immediately from the standard identities \cite[pp.~20, 21, 49]{zagier}
\begin{align}
\label{eq:moddisc-prod}
E_4^3 - E_6^2 &= 1728 q^2 \prod_{n=1}^\infty (1-q^{2n})^{24},
\\
\label{eq:e2e4minuse6}
E_2 E_4 - E_6
 & = \frac{3}{2\pi i} \frac{d E_4}{dz} = 720 \sum_{n=1}^\infty n \sigma_3(n) q^{2n},
\end{align}
which imply that both $E_4^3 - E_6^2$ and $E_2 E_4 - E_6$ take positive real values on the positive imaginary  axis.

For the claim about $\psi_I(it)$, recall 
Jacobi's identity $\theta_2^4+\theta_4^4 = \theta_3^4$ (see \cite[p.~28]{zagier}), and set $\lambda(z) = \theta_2^4/\theta_3^4 = 1-\theta_4^4/\theta_3^4$ (the modular lambda function \cite[p.~63]{zagier}). It is clear from these defining relations of $\lambda(z)$ that for $t>0$, $\lambda(it)$ takes real values in $(0,1)$.
Now note that
\begin{align*}
\frac{1}{128} \psi_I
& =
\frac{\theta_3^4 + \theta_4^4}{\theta_2^8} + 
\frac{\theta_4^4 - \theta_2^4}{\theta_3^8}
= 
\frac{1}{\theta_3^4} \cdot \frac{ \theta_3^8 + \theta_3^4 \theta_4^4}{\theta_2^8} + 
\frac{1}{\theta_3^4} \cdot \frac{\theta_4^4 - \theta_2^4}{\theta_3^4}
\\ & =
\frac{1}{\theta_3^4} \left( \frac{1}{\lambda^2} + \frac{1}{\lambda} \cdot \frac{1-\lambda}{\lambda}
+ (1-\lambda) - \lambda
\right)
= 
\frac{1}{\theta_3^4} \frac{(1-\lambda)(2 + \lambda + 2\lambda^2)}{\lambda^2}.
\end{align*}
Since the function $x\mapsto \frac{(1-x)(2+x+2x^2)}{x^2}$ is positive for $x \in (0,1)$, and since $\theta_3(it)^4 > 0$ for $t>0$, we get the claim that $\psi_I(i t) > 0$.
\qed

\section{Proof of \eqref{viaz-ineq2}}

We will make use of the standard modular transformation properties \cite[pp.~996--997]{viazovska}
\begin{align}
\label{eq:transrel-theta2}
\theta_2(z+1)^4 & = - \theta_2(z)^4, 
\qquad 
\theta_2(-1/z)^4 = -z^2 \, \theta_4(z)^4, \\
\label{eq:transrel-theta3}
\theta_3(z+1)^4 &= \theta_4(z)^4, \phantom{-}
\qquad 
\theta_3(-1/z)^4 = -z^2 \, \theta_3(z)^4, \\
\label{eq:transrel-theta4}
\theta_4(z+1)^4 &= \theta_3(z)^4, \phantom{-}
\qquad
\theta_4(-1/z)^4 = -z^2 \, \theta_2(z)^4,
\\
\label{eq:transrel-e2}
E_2(z+1) & = E_2(z),
\qquad \quad\,
E_2(-1/z) = z^2 E_2(z) - \frac{6iz}{\pi},
\\
\label{eq:transrel-e4}
E_4(z+1) & = E_4(z),
\qquad \quad\,
E_4(-1/z) = z^4 E_4(z),
\\
\label{eq:transrel-e6}
E_6(z+1) & = E_6(z),
\qquad \quad\,
E_6(-1/z) = z^6 E_6(z).
\end{align}
Using \eqref{eq:transrel-e2}--\eqref{eq:transrel-e6}, a simple calculation shows that
$$
z^2 \phi_0(-1/z) = 
 1728 \left[ \frac{(E_2E_4-E_6)^2}{E_4^3 - E_6^2 } z^2
- \frac{12i}{\pi} \cdot
 \frac{E_4 (E_2E_4-E_6)}{E_4^3 - E_6^2 }  z
- \frac{36}{\pi^2}
\left( \frac{E_4^2}{E_4^3 - E_6^2 } \right)
\right].
$$
(This is a slightly simplified version of Eq.~(29) from \cite{viazovska}.)
Similarly, with the help of \eqref{eq:transrel-theta2}--\eqref{eq:transrel-theta4} we see that
$$
z^2 \psi_I(-1/z) = 
-128 \left( \frac{\theta_3^4 + \theta_2^4}{\theta_4^8} +
\frac{\theta_2^4 - \theta_4^4}{\theta_3^8} \right).
$$
We will separate the proof of \eqref{viaz-ineq2} into two parts, proving separately that
$$ B(t) > 0 \ \ \textrm{for }t\ge1 \quad \textrm{and} \quad
t^2 B(1/t) > 0 \ \ \textrm{for }t\ge1, $$ 
that is, equivalently, that
$$
\frac{\pi^2}{36} t^2 \phi_0(i/t) < \psi_I(it)
 \ \ \textrm{for }t\ge1 \quad \textrm{and} \quad
\frac{\pi^2}{36} \phi_0(it) < t^2 \psi_I(i/t)
 \ \ \textrm{for }t\ge1.
$$
It is convenient to clear the denominators in each of these inequalities by multiplying both sides by $E_4^3-E_6^2$, which is also equal to $\frac{27}{4} (\theta_2 \theta_3 \theta_4)^8$ by a well-known identity.  \cite[p.~29]{zagier}
We therefore define
\begin{align}
\label{eq:viazineq-def-ftilde}
f(z) & = 
 \frac{1}{864}\cdot \frac{\pi^2}{36}
 (E_4^3-E_6^2) \phi_0(z) =
\frac{\pi^2}{18} (E_2 E_4 - E_6)^2, 
\\
\widetilde{f}(z) & = -\frac{1}{864}\cdot \frac{\pi^2}{36} (E_4^3-E_6^2) z^2 \phi_0(-1/z) 
\nonumber \\ & =
-\frac{\pi^2}{18} (E_2 E_4-E_6)^2 z^2
+ \frac{2\pi i}{3} E_4(E_2 E_4-E_6) z
+ 2 E_4^2,
\label{eq:viazineq-def-f}
\\
g(z) &= 
-\frac{1}{864} 
(E_4^3-E_6^2)
z^2 \psi_I(-1/z)
= \theta_2^8 ( \theta_3^{12} + \theta_2^4 \theta_3^8 + \theta_2^4 \theta_4^8 - \theta_4^{12} ).
\\
\widetilde{g}(z) &= 
\frac{1}{864} 
(E_4^3-E_6^2)
\psi_I(z)
=
\theta_4^8 ( \theta_3^{12} + \theta_4^4 \theta_3^8 + \theta_2^8 \theta_4^4 - \theta_2^{12} ),
\label{eq:viazineq-def-g}
\end{align}
By the above remarks, in order to deduce \eqref{viaz-ineq2} it will be sufficient to prove the following inequalities:
\begin{align}
\label{eq:viazovskas-ineq-ftildegtilde}
f(it) & < g(it) \qquad \textrm{for }t\ge1, \tag{V2-I}
\\
\label{eq:viazovskas-ineq-fg}
\widetilde{f}(it) &< \widetilde{g}(it) \qquad \textrm{for }t\ge1. \tag{V2-II}
\end{align}
As a final bit of preparation, recall the known explicit evaluations
\begin{align}
E_2(i) &= \frac{3}{\pi},
\qquad \qquad\ 
E_4(i) = \frac{ 3\Gamma(1/4)^8 }{64 \pi^6},
\qquad
E_6(i) = 0,
\label{eq:eisen-exp-evaluations}
\\
\theta_2(i) &= \frac{\Gamma(1/4)}{(2\pi)^{3/4}},
\qquad
\theta_3(i) = \frac{\Gamma(1/4)}{\sqrt{2}\, \pi^{3/4}},
\qquad \ \ \,
\theta_4(i) = \frac{\Gamma(1/4)}{(2\pi)^{3/4}}.
\label{eq:theta-exp-evaluations}
\end{align}
Here, $\Gamma(\cdot)$ denotes the Euler gamma function. (The numerical value of $\Gamma(1/4)$ is approximately $3.62561$. \cite{oeis-gamma14})
For the proof of \eqref{eq:theta-exp-evaluations}, see \cite[p.~325]{ramanujans-notebook5}, \cite[eq.~(2.21), p.~307]{cox}. The identity $E_2(i)= 3/\pi$ is an immediate consequence of~\eqref{eq:transrel-e2}. The relation $E_6(i)= 0$ is proved in \cite[p.~40]{apostol}, and the formula for $E_4(i)$ follows from 
\eqref{eq:theta-exp-evaluations} and the identity $E_4 = \frac12(\theta_2^8+\theta_3^8+\theta_4^8)$, proved, e.g., in \cite[p.~29]{zagier}; see also \cite[p.~290]{tsumura}.

\subsection{Proof of \eqref{eq:viazovskas-ineq-ftildegtilde}}

The functions $f(z)$, $g(z)$ have Fourier expansions
\begin{align}
f(z) &=
28800 \pi ^2 q^4 +
1036800 \pi ^2 q^6 + 
14169600 \pi ^2 q^8 +
\ldots
=: \sum_{n=4}^\infty a_n q^n,
\label{eq:viaz-ineq-FQ}
\\
g(z) & =
20480 q^3 + 
2015232 q^5 +
41656320 q^7 +
\ldots =: \sum_{n=3}^\infty b_n q^n.
\label{eq:viaz-ineq-GQ}
\end{align}
The coefficients $a_n$ in \eqref{eq:viaz-ineq-FQ} are nonnegative: this is immediate from \eqref{eq:e2e4minuse6}. Similarly, we have $b_n \ge 0$ for all $n$. To see this, 
let $\gamma(z) = \theta_2^8 \theta_3^{12} + \theta_2^{12} \theta_3^8$, and
observe that, by \eqref{eq:transrel-theta2}--\eqref{eq:transrel-theta4}, $g(z)$ can be represented as
\begin{equation}
\label{eq:Gtilde-gamma}
g(z) = \gamma(z) - \gamma(z+1).
\end{equation}
The Fourier coefficients of $\gamma$ are manifestly nonnegative, and, since the substitution $z\mapsto z+1$ corresponds to replacing each occurrence of $q$ by $-q$ in the Fourier series, the relationship \eqref{eq:Gtilde-gamma} means that the Fourier expansion of $g$ consists of twice the odd terms in the Fourier expansion of~$\gamma$, and therefore also has nonnegative coefficients.

From the above remarks it now follows that the function 
$t \mapsto e^{3\pi t} f(it) = \sum_{n=4}^\infty a_n e^{-\pi (n-3) t}
$ is a nonincreasing function of $t$. 
Using \eqref{eq:eisen-exp-evaluations}, we then get for all $t \ge1$ the bound
\begin{equation}
e^{3\pi t}f(it) \le e^{3\pi} f(i)
=
e^{3\pi} 
\frac{\pi^2}{18} \left( \frac{3}{\pi} \frac{3 \Gamma(1/4)^8}{64 \pi^6}-0 \right)^2
=
e^{3\pi} \frac{9\Gamma(1/4)^{16} }{8192\,  \pi^{12}} \approx 13130.47.
\label{eq:e3pitfit}
\end{equation}
On the other hand, by \eqref{eq:viaz-ineq-GQ} and the observation about the nonnegativity of the coefficients $b_n$, the bound
$e^{3 \pi t} g(it) = 20480 + \sum_{n=4}^\infty b_n e^{-\pi (n-3)t} \ge 20480$
holds for all $t>0$.
Combining this with \eqref{eq:e3pitfit} gives~\eqref{eq:viazovskas-ineq-ftildegtilde}. \qed

\subsection{Proof of \eqref{eq:viazovskas-ineq-fg}}

In a similar vein, we examine the $q$-series expansions of $\widetilde{f}(z)$, $\widetilde{g}(z)$ and their properties. From \eqref{eq:viazineq-def-f} and \eqref{eq:viazineq-def-g}, we obtain expansions of the forms
\begin{align}
\widetilde{f}(z) & = 
2 + (480 \pi i z + 960) q^2 + (-28800 \pi^2 z^ 2 + 123840 \pi i z + 123840) q^4 
\nonumber \\ &\quad
+ (-1036800 \pi^2 z^2 + 3150720 \pi i z + 2100480) q^6
+ \ldots
=: \sum_{n=0}^\infty c_n(z) q^n,
\label{eq:ftilde-fourier}
\\
\widetilde{g}(z) & = 2 + 240 q^2 - 10240 q^3 + 134640 q^4 - 1007616 q^5 + \ldots
=: \sum_{n=0}^\infty d_n q^n.
\label{eq:gtilde-fourier}
\end{align}
Here, \eqref{eq:gtilde-fourier} is a conventional Fourier series, whereas \eqref{eq:ftilde-fourier} is a more unusual expansion in powers of $q = e^{\pi i z}$ in which each coefficient $c_n(z)$ is itself a quadratic polynomial in $z$. It is convenient to renormalize these expressions, defining new functions
\begin{align}
\widetilde{F}(z) & = -\frac{\widetilde{f}(z)-2}{q^2} = -\sum_{n=2}^\infty c_n(z) q^{n-2}
\nonumber \\ & = (-480 \pi i z - 960) + (28800 \pi^2 z^ 2 - 123840 \pi i z - 123840) q^2 + \ldots,
\label{eq:def-kofz}
\\
\widetilde{G}(z) & = -\frac{\widetilde{g}(z)-2}{q^2} = -\sum_{n=2}^\infty d_n q^{n-2}
= -240 + 10240 q - 134640 q^2 + 1007616 q^3 + \ldots.
\label{eq:def-lofz}
\end{align}
The inequality \eqref{eq:viazovskas-ineq-fg} can now be restated as the claim that $\widetilde{G}(it) < \widetilde{F}(it)$ for all~$t\ge1$. This will follow from the combination of the following two lemmas.

\begin{lem} 
\label{lem:lit-ineq}
$\widetilde{G}(it) \le 288$ for all $t\ge1$.
\end{lem}

\begin{lem} 
\label{lem:kit-ineq}
$\widetilde{F}(it) \ge 468$ for all $t\ge1$.
\end{lem}

The following auxiliary claim will be used in the proof of Lemma~\ref{lem:lit-ineq}.

\begin{lem} 
\label{lem:altsigns}
We have $(-1)^n d_n \ge 0$ for $n\ge0$.
\end{lem}

\begin{proof}
By \eqref{eq:transrel-theta2}--\eqref{eq:transrel-theta4}, the function $\widetilde{g}(z+1) = \sum_{n=0}^\infty (-1)^n d_n q^n$ can be written as
\begin{equation}
\label{eq:tildegzplus1}
\widetilde{g}(z+1) =
\theta_3^{12} \theta_2^8 + \theta_3^8 \theta_2^{12} + \theta_3^{12} \theta_4^8 + \theta_3^8 \theta_4^{12}.
\end{equation}
The claim is that the Fourier series of this function has nonnegative coefficients.
This fact was proved by Slipper \cite[p.~76]{slipper}, who deduced it from a certain identity representing the function on the right-hand side of \eqref{eq:tildegzplus1} in terms of the theta series of a certain 20-dimensional lattice, known as \nobreak{``DualExtremal(20,2)a''}. Here is a self-contained proof that only uses elementary properties of the thetanull functions.
Denote for convenience
$$
Z = \theta_3^4, \qquad X = \theta_2^4, \qquad Y = 2Z-X.
$$
Then $X$ and $Z$ have Fourier series with nonnegative coefficients, and,
again recalling the identity $\theta_2^4+\theta_4^4 = \theta_3^4$,
we see that $Y = \theta_3^4 + \theta_4^4 = \theta_3(z)^4 + \theta_3(z+1)^4$ (recall \eqref{eq:transrel-theta3} above), so the Fourier series of $Y$ also has nonnegative coefficients.
Now,
observe that $\widetilde{g}(z+1)$ can be expressed as
\begin{align*}
\widetilde{g}(z+1) & = 
Z^3 X^2 + Z^2 X^3 + Z^3 (Z-X)^2 + Z^2 (Z-X)^3
\\ & 
= 
\frac{1}{16}\left( 6 X^5 + 15 X^4 Y + 10 X^3 Y^2 + Y^5 \right),
\end{align*}
and therefore also has nonnegative Fourier coefficients.
\end{proof}

\begin{proof}[Proof of Lemma~\ref{lem:lit-ineq}]
Define 
\begin{equation}
\label{eq:defH}
H(z) = \frac{\widetilde{G}(z) - \widetilde{G}(z+1)}{2} = \sum_{m=1}^\infty (-d_{2m+1}) q^{2m-1}
= 10240 q + 10007616 q^3 + \ldots .
\end{equation}
Two crucial properties of $H(z)$ are: (a) the function $t\mapsto H(it) = 
\sum_{m=1}^\infty (-d_{2m+1}) e^{-\pi (2m-1)t}
$ is nonincreasing (each summand is nonincreasing, by Lemma~\ref{lem:altsigns}); and (b) $\widetilde{G}(it)+240 \le H(it)$ for all $t>0$ (this follows from Lemma~\ref{lem:altsigns} together with the observation that the constant coefficient in~\eqref{eq:def-lofz} is $-240$).
Now note that, by \eqref{eq:transrel-theta2}--\eqref{eq:transrel-theta4}, \eqref{eq:viazineq-def-g}, and \eqref{eq:def-lofz}, $H(z)$ can be expressed explicitly as
\begin{align*}
H(z) & =
- \frac12 q^{-2} 
\Big[ \theta_4^8 ( \theta_3^{12} + \theta_4^4 \theta_3^8 + \theta_2^8 \theta_4^4 - \theta_2^{12}) - 2
- \theta_3^8 ( \theta_4^{12} + \theta_3^4 \theta_4^8 + \theta_2^8 \theta_3^4 + \theta_2^{12} ) + 2
\Big]
\\ & =
\frac12 q^{-2} \left(
\theta_2^8 \theta_3^{12}
+ \theta_2^{12} \theta_3^8 + \theta_2^{12} \theta_4^8 - \theta_2^8 \theta_4^{12}
\right)
= 
\frac12 q^{-2} \left(
\theta_2^8 (\theta_3^{12}-\theta_4^{12})
+ \theta_2^{12} (\theta_3^8 + \theta_4^8)
\right).
\end{align*}
Therefore using the evaluations \eqref{eq:theta-exp-evaluations} we get that for all $t\ge1$,
\begin{align*}
\widetilde{G}(it) & 
\le 
-240 + H(it)
\le -240 + H(i)
\\ & = -240 + 
\frac12 e^{2\pi} 
\left(\frac{\Gamma(1/4)}{(2\pi)^{3/4}}\right)^{20}
\left( (2^{1/4})^{12} - 1 + (2^{1/4})^8 + 1 \right)
\\ & = 
-240 + \frac12 e^{2\pi} 
\frac{\Gamma(1/4)^{20}}{(2\pi)^{15}}(8+4)
=
-240 + 6 e^{2\pi} 
\frac{\Gamma(1/4)^{20}}{(2\pi)^{15}}
\approx 287.02,
\qquad \textrm{as claimed.} \qedhere
\end{align*}
\end{proof}

\begin{proof}[Proof of Lemma~\ref{lem:kit-ineq}]
We strategically separate $\widetilde{F}(z)$ into three components, defining
\begin{align}
\widetilde{F}_1(z) &= -480 \pi i z + (28800 \pi ^2 z^2 - 123840 \pi i z - 123840) q^2, 
\\
\widetilde{F}_2(z) & =
\frac{\pi^2}{18} q^{-2} (E_2 E_4-E_6)^2 z^2
- 2 q^{-2} (E_4^2 - 1)
+ (-28800\pi^2 z^2 + 123840) q^2,
\label{eq:K2-def}
\\
\widetilde{F}_3(z) &= 
-\frac{2\pi i}{3} q^{-2} E_4(E_2 E_4 - E_6) z + (480 \pi i z + 123840 \pi i z q^2),
\label{eq:K3-def}
\end{align}
so that, by \eqref{eq:viazineq-def-f} and \eqref{eq:def-kofz}, we have
\begin{equation}
\label{eq:KK1K2K3}
\widetilde{F}(z) = \widetilde{F}_1(z) + \widetilde{F}_2(z) + \widetilde{F}_3(z).
\end{equation}

We now make the following elementary observations:
\begin{enumerate}[(a)]

\item The function $t\mapsto \widetilde{F}_1(it)$ is monotone increasing on $[1,\infty)$, 

\smallskip

\noindent
\textit{Proof.}
Assume that $t\ge 1$. A trivial calculation gives that
\begin{align*}
\frac{d}{dt} \Big(\widetilde{F}_1(it)\Big) &
= 
480\pi e^{-2\pi t} \left( e^{2\pi t} + 120 \pi^2 t^2 - 636 \pi t + 774 \right)
\\ & 
\ge
480\pi e^{-2\pi t} \left( e^{2\pi} + 120 \pi^2 t^2 - 636 \pi t + 774 \right).
\end{align*}
The last expression is of the form $e^{-2\pi t}$ times a quadratic polynomial in~$t$, which, it is easy to check, is positive on the real line. Thus, we have shown that $\widetilde{F}_1'(t)>0$ for $t\ge1$, which proves the claim.

\medskip

\item The function $t\mapsto \widetilde{F}_2(it)$ is monotone increasing on $[1,\infty)$.

\smallskip

\noindent \textit{Proof.}
Let $(\alpha_n)_{n=2}^\infty$ and $(\beta_n)_{n=1}^\infty$ be the coefficients in the Fourier series
$$
(E_2 E_4-E_6)^2 = \sum_{n=2}^\infty \alpha_n q^{2n}, \qquad 
E_4^2-1 = \sum_{n=1}^\infty \beta_n q^{2n}.
$$
Clearly $\alpha_n \ge 0$ (see \eqref{eq:e2e4minuse6}) and $\beta_n \ge  0$ for all $n$. One can also easily check that $\alpha_2 = 518400$ and~$\beta_2 = 61920$.
Then, on inspection of \eqref{eq:K2-def}, we see that
$$
\widetilde{F}_2(it) = -2\beta_1 + \sum_{n=3}^\infty \left(-\frac{\pi^2}{18} \alpha_n t^2 - 2\beta_n\right) e^{-\pi (2n-2) t}.
$$
(The summand associated with $n=2$ is precisely cancelled out by the term $(-28800\pi^2 z^2 + 123840) q^2$ in \eqref{eq:K2-def}.)
Now for each $n\ge2$, the $n$th summand in this series is easily seen to be an increasing function of $t$ for $t\ge \frac{1}{(n-1)\pi}$, so in particular for $t\ge 1$. Thus $t\mapsto \widetilde{F}_2(it)$ is also increasing for $t\ge1$.

\medskip

\item $\widetilde{F}_3(i t)\ge 0$ for all $t>0$.

\smallskip

\noindent \textit{Proof.} 
Let $(\delta_n)_{n=1}^\infty$ be the coefficients in the Fourier series 
$
E_4(E_2 E_4-E_6) = \sum_{n=1}^\infty \delta_n q^{2n}.
$
Then $\delta_n\ge 0$ for all $n$, and
we have $\delta_1= 720$ and $\delta_2= 185760$.
Referring to \eqref{eq:K3-def}, we then see that
$$ 
\widetilde{F}_3(it) = \frac{2\pi t}{3} \sum_{n=3}^\infty \delta_n e^{-\pi (2n-2)t} \ge 0,
$$
since the summands associated with $n=1,2$ are cancelled by the term $(480 \pi i z + 123840 \pi i z q^2)$ in \eqref{eq:K3-def}.

\end{enumerate}

\bigskip
Finally, combining \eqref{eq:KK1K2K3} with the observations (a)--(c) above, we get that for $t\ge 1$,
\begin{align*}
\widetilde{F}(it) &\ge \widetilde{F}_1(i t) + \widetilde{F}_2(i t) \ge \widetilde{F}_1(i) + \widetilde{F}_2(i) 
\\ & =
480\pi + 123840 e^{-2\pi} + e^{2\pi} \Big(
-\frac{\pi^2}{18}  (E_2(i)E_4(i)-E_6(i))^2 - 2(E_4(i)^2-1)
\Big)
\\ & =
480\pi + 123840 e^{-2\pi} + e^{2\pi} \left(
2- \frac{45 \Gamma(1/4)^{16}}{8192 \pi^{12}}
\right) \approx 468.39, 
\qquad \textrm{as claimed.} \qedhere
\end{align*}
\end{proof}

\bibliographystyle{plain}
\bibliography{ineq.bib}

\end{document}